\documentclass[a4paper,12pt,reqno,psamsfonts]{amsart}
\usepackage[headings]{fullpage}             
\usepackage{hyperref}                       

\usepackage[all]{xy}                        

\usepackage{amsmath}
\usepackage{amssymb}
\usepackage{amsthm}

\theoremstyle{plain}
\newtheorem{thm}{Theorem}[section]
\newtheorem*{thm*}{Theorem}

\newtheorem{lem}[thm]{Lemma}
\newtheorem{cor}[thm]{Corollary}
\newtheorem*{cor*}{Corollary}

\newtheorem*{claim*}{Claim}

\theoremstyle{definition}
\newtheorem{defn}[thm]{Definition}

\theoremstyle{remark}
\newtheorem*{rem*}{Remark}
\newtheorem*{ex*}{Example}
\newtheorem*{quest*}{Question}
\newtheorem*{ackn*}{Acknowledgements}

\theoremstyle{plain}
\newtheorem*{comp_ax*}{Comparison axiom}
\newtheorem*{ext_ax*}{Extension axiom}
\theoremstyle{remark}
\newtheorem*{summ*}{Summary}

\newcommand{\Vol}{\operatorname{Vol}}
\newcommand{\stabsys}{\operatorname{stabsys}}
\newcommand{\sys}{\operatorname{sys}}

\begin{document}


\title{On manifolds satisfying stable systolic inequalities}
\author{Michael Brunnbauer}
\address{Mathematisches Institut, Ludwig-Maximilians-Universit\"at M\"unchen, Theresienstr.\ 39, D-80333 M\"unchen, Germany}
\email{michael.brunnbauer@mathematik.uni-muenchen.de}
\date{\today}

\begin{abstract}
We show that for closed orientable manifolds the $k$-dimensional stable systole admits a metric-independent volume bound if and only if there are cohomology classes of degree $k$ that generate cohomology in top-degree. Moreover, it turns out that in the nonorientable case such a bound does not exist for stable systoles of dimension at least two.

Additionally, we prove that the stable systolic constant depends only on the image of the fundamental class in a suitable Eilenberg-Mac\,Lane space. Consequently, the stable $k$-systolic constant is completely determined by the multilinear intersection form on $k$-dimensional cohomology.
\end{abstract}

\maketitle


\section{Introduction}\label{intro}

Let $M$ be a connected closed smooth manifold of dimension $n$, and let $g$ be a Riemannian metric on it. Let $1\leq k\leq n-1$ be an integer. The volume of an integral or real $k$-dimensional Lipschitz cycle $c=\sum_i r_i\sigma_i$ is given by
\[ \Vol_k(c) := {\textstyle\sum_i} |r_i| \Vol_k(\Delta^k,\sigma_i^* g). \]
To see that this is well-defined note that the pullback `metric' $\sigma_i^*g$ is almost everywhere defined by Rademacher's theorem and is positive semidefinite. Thus, it has an almost everywhere defined `volume form', and $\Vol_k(\Delta^k,\sigma_i^* g)$ is the integral of this $k$-form over $\Delta^k$.

For an integral homology class $\alpha\in H_k(M;\mathbb{Z})$ the \emph{volume} $\Vol_k(\alpha)$ is defined as the infimum of the volumes of all integral Lipschitz cycles representing $\alpha$. The \emph{stable norm} $\|\alpha\|$ of a real homology class $\alpha \in H_k(M;\mathbb{R})$ is defined in the same way but using all real Lipschitz cycles representing $\alpha$. Federer showed that the stable norm is in fact a norm, and moreover that  
\[ \|\alpha\| = \lim_{i\to\infty} {\textstyle\frac{1}{i}}\Vol_k(i\alpha) \]
holds for all integral homology classes $\alpha$. (This is proved in \cite{Federer(1974)}, sections 4 and 5. See also \cite{Gromov(1999)}, section 4.C.).

The \emph{stable $k$-systole} $\stabsys_k(M,g)$ is defined as the minimum of the stable norm on the nonzero classes of the integral lattice $H_k(M;\mathbb{Z})_\mathbb{R}$ in $H_k(M;\mathbb{R})$. The main focus of the present article is on the existence and nonexistence of stable systolic inequalities of the form
\begin{equation}\label{stabsys_ineq}
\stabsys_k(M,g)^{n/k} \leq C(M) \cdot \Vol(M,g),
\end{equation}
in which the constant $C(M)$ does not depend on the metric $g$. Therefore, it is natural to look at the \emph{stable $k$-systolic constant}
\[ \sigma_k^{st} (M) := \inf_g \frac{\Vol(M,g)}{\stabsys_k(M,g)^{n/k}}\;, \]
where the infimum is taken over all Riemannian metrics $g$ on $M$. (If $H_k(M;\mathbb{Z})_\mathbb{R}$ is trivial, the stable $k$-systolic constant is understood as zero.) Obviously, $\sigma_k^{st}(M) > 0$ if and only if $M$ satisfies a stable systolic inequality \eqref{stabsys_ineq}. Moreover, the reciprocal of $\sigma_k^{st}(M)$ is the best constant $C(M)$ such that \eqref{stabsys_ineq} holds.

For example, it is known by work of Gromov that
\[ \sigma_2^{st} (\mathbb{C}\mathrm{P}^n) = 1/n!\;, \]
see \cite{Gromov(1999)}, Theorem 4.36. (A more detailed proof may be found in \cite{Katz(2007)}, section 13.2.) In fact, this is the only example of a higher-dimensional stable systolic constant whose value is known and not zero.

In his Filling paper, Gromov gave a sufficient condition for an orientable manifold to satisfy a more general stable systolic inequality.

\begin{thm}[\cite{Gromov(1983)}, 7.4.C]\label{thm_gromov}
Let $M$ be a connected closed orientable manifold of dimension $n$, and let $(k_1,\ldots, k_p)$ be a partition of $n$, i.\,e.\ an unordered sequence of positive integers such that $n=\sum_{i=1}^p k_i$. If there are cohomology classes $\beta_i \in H^{k_i} (M;\mathbb{R})$ such that their cup product $\beta_1\smile\cdots\smile\beta_p\in H^n(M;\mathbb{R})$ does not vanish, then 
\[ \prod_{i=1}^p \stabsys_{k_i}(M,g) \leq C \cdot \Vol(M,g) \]
for a constant $C>0$ depending only on the dimension $n$, the partition $(k_1,\ldots,k_p)$, and the Betti numbers $b_{k_i}(M)$ of $M$.
\end{thm}

For the dependence on the Betti numbers and the partition the reader is referred to \cite{BK(2003)}, Theorem 2.1. Applied to stable systolic inequalities \eqref{stabsys_ineq}, i.\,e.\ to the case of partitions $(k,\ldots,k)$, this theorem implies that the stable $k$-systolic constant $\sigma_k^{st}(M)$ of a connected closed orientable manifold of dimension $n=kp$ is positive if there are cohomology classes $\beta_1,\ldots,\beta_p \in H^k (M;\mathbb{R})$ such that $\beta_1\smile\cdots\smile\beta_p\neq 0$. (See also \cite{BK(2003)}, Theorem 2.7.) 

For example, cohomologically symplectic manifolds (i.\,e.\ even-dimensional manifolds $M^{2n}$ possessing a cohomology class $\omega\in H^2(M;\mathbb{R})$ such that $\omega^n\neq 0$) have nonvanishing stable systolic constants in all even dimensions that divide the dimension of $M$.

In \cite{Babenko(1992)}, Theorem 8.2 (c), Babenko showed that in the case $k=1$ the condition stated above is also necessary for the existence of a stable $1$-systolic inequality. More precisely, he proved that if the Jacobi mapping $\Phi:M\to T^b$, with $b:=b_1(M)$ the first Betti number, maps the fundamental class $[M]_\mathbb{Z}$ of the manifold to zero, then $\sigma_1^{st}(M)=0$. (Throughout this article the subscript will indicate the coefficients of the homology group in which the fundamental class lives. This is necessary because later on we will also consider coefficients in $\mathbb{Z}_2$ and in the orientation bundle.) But if $\Phi_*[M]_\mathbb{Z}\neq 0$, then there are cohomology classes $\beta'_1,\ldots,\beta'_n\in H^1(T^b;\mathbb{R})$ such that 
\[ \langle \beta'_1 \smile\cdots\smile \beta'_n, \Phi_*[M]_\mathbb{Z} \rangle\neq 0 \]
since the cohomology of the torus is generated by classes of degree one. Therefore, the cohomology classes $\beta_i := \Phi^*\beta'_i \in H^1(M;\mathbb{R})$ have nonvanishing product, and the stable $1$-systolic constant of $M$ is nonzero by Gromov's theorem.

Extending this, we will show the following equivalence, which is valid for all integers $1\leq k \leq n-1$.

\begin{thm}\label{stabsys_thm}
Let $M$ be a connected closed orientable manifold of dimension $n$. The stable $k$-systolic constant $\sigma_k^{st}(M)$ does not vanish if and only if $n$ is a multiple of $k$, say $n=kp$, and there exist cohomology classes $\beta_1,\ldots,\beta_p \in H^k(M;\mathbb{R})$ such that $\beta_1\smile\cdots\smile\beta_p\neq 0$ in $H^n(M;\mathbb{R})$.
\end{thm}

For instance, by Poincar\'e duality, the stable middle-dimensional systolic constant of even-dimensional manifolds vanishes if and only if the middle-dimensional Betti number is zero.

For nonorientable manifolds we will prove a kind of `general stable systolic freedom' whenever no one-dimensional systoles are involved. (The term \emph{(stable) systolic freedom} refers to the absence of a (stable) systolic inequality, see e.\,g.\ the articles \cite{KS(1999)}, \cite{KS(2001)}, and \cite{Babenko(2002)} for other kinds of systolic freedom.)

\begin{thm}\label{freedom}
Let $M$ be a connected closed nonorientable manifold of dimension $n$.  Let $p_k\geq 0$ be nonnegative real numbers for $k=2,\ldots,n-1$. Then there is no constant $C>0$ such that
\[ \prod_{k=2}^{n-1} \stabsys_k(M,g)^{p_k} \leq C \cdot \Vol(M,g) \]
holds for all Riemannian metrics $g$ on $M$. In particular, $\sigma_k^{st}(M)=0$ for $k=2,\ldots,n-1$.
\end{thm}

Denote by $b:=b_k(M)$ the $k$-th Betti number. There exists a map $\Phi: M\to K(\mathbb{Z}^b,k)$ that induces an isomorphism
\[ H_k(M;\mathbb{Z})_\mathbb{R} \xrightarrow{\cong} H_k(K(\mathbb{Z}^b,k);\mathbb{Z}). \]
To see this, note that by the canonical isomorphism
\[ [M, K(\mathbb{Z}^b,k) ] \cong H^k (M;\mathbb{Z})^b \]
it suffices to choose classes $\beta_1,\ldots,\beta_b \in H^k(M;\mathbb{Z})$ that represent a basis of $H^k(M;\mathbb{Z})_\mathbb{R}$. This choice corresponds to a map
\[ \Phi:M\to K(\mathbb{Z}^b,k) \]
such that the canonical basis $\delta_1,\ldots,\delta_b$ of $H^k(K(\mathbb{Z}^b,k);\mathbb{Z})$ is pulled back to $\beta_i=\Phi^*\delta_i$. Thus, $\Phi$ induces an isomorphism on $k$-dimensional cohomology modulo torsion and consequently also on the integral lattices of $k$-dimensional homology. Note however that the homotopy class of this map is not uniquely determined by the isomorphism $H_k(M;\mathbb{Z})_\mathbb{R} \cong H_k(K(\mathbb{Z}^b,k);\mathbb{Z})$ except when $H^k(M;\mathbb{Z})$ is torsion-free, which happens for example if $k=1$.

As in \cite{Brunnbauer(2007b)} for asymptotic invariants and the (stable) $1$-systolic constant, comparison and extension techniques will allow us to prove homological invariance of stable systolic constants.

\begin{thm}\label{homol_invar_thm}
Let $M$ and $N$ be two connected closed orientable manifolds of dimension $n$, and let $1\leq k\leq n-1$. Suppose that $b_k(M) = b_k(N) =:b$ and that there are maps $\Phi:M\to K(\mathbb{Z}^b,k)$ and $\Psi:N\to K(\mathbb{Z}^b,k)$ such that the induced homomorphisms on the integral lattices of $k$-dimensional homology are bijective and such that 
\[ \Phi_*[M]_\mathbb{Z} = \Psi_*[N]_\mathbb{Z} \in H_n(K(\mathbb{Z}^b,k);\mathbb{R}). \]
Then the stable $k$-systolic constants coincide: $\sigma_k^{st}(M) = \sigma_k^{st}(N)$.
\end{thm}

As a direct consequence two manifolds have the same stable $k$-systolic constants whenever there exists a degree one mapping between them that induces an isomorphism of the integral lattices of $k$-dimensional homology.

Let $M$ be a connected closed oriented manifold of dimension $n=kp$. Consider the \emph{multilinear intersection form}
\begin{align*}
Q_M^k: (H^k(M;\mathbb{Z})_\mathbb{R})^p &\to \mathbb{Z} \\
(\beta_1,\ldots,\beta_p) &\mapsto \langle\beta_1\smile\cdots\smile\beta_p,[M]_\mathbb{Z} \rangle.
\end{align*}
By Theorem \ref{stabsys_thm}, this form vanishes identically if and only if $\sigma^{st}_k(M)=0$.

Using the computation of the real cohomology ring of the Eilenberg-Mac\,Lane space $K(\mathbb{Z},k)$ by Cartan and Serre, we are able to derive the following corollary of Theorem \ref{homol_invar_thm}.

\begin{cor}\label{intersection_form}
Let $M$ and $N$ be two connected closed orientable manifolds of dimension $n=kp$. If the multilinear intersection forms $Q^k_M$ and $Q^k_N$ are equivalent over $\mathbb{Z}$, then
\[ \sigma_k^{st}(M) = \sigma_k^{st}(N). \]
\end{cor}

In the next section, we will give an axiomatic approach to stable systolic constants. This will be used to prove Theorem \ref{homol_invar_thm} in section \ref{homol_invar}. In section \ref{top_lemmata}, we recall two lemmata from \cite{Brunnbauer(2007b)} that are essential for various proofs in the later sections. Theorems \ref{stabsys_thm} and \ref{freedom} will be proved in section \ref{existence} using some topological facts on spheres and their loop spaces that are presented in section \ref{spheres_loop_sp}.

The remaining case of nonorientable manifolds and stable $1$-systoles was already investigated in \cite{Brunnbauer(2007b)}. The respective results will be mentioned in the different sections.

As general references to systolic geometry, we would like to mention chapter 7.2 of Berger's book \cite{Berger(2003)}, the survey article \cite{CK(2003)}, and the book \cite{Katz(2007)}.

\begin{ackn*}
I would like to thank D.\,Kotschick for his continuous advice and help. I am also grateful to M.\,Katz for bringing my attention to stable systolic inequalities and for many useful remarks. Moreover, I thank the anonymous referee for her/his helpful remarks. Financial support from the \emph{Deutsche Forschungsgemeinschaft} is gratefully acknowledged.
\end{ackn*}

\section{Axioms for stable systolic constants}

Using continuous piecewise smooth Riemannian metrics, the definition of stable systoles extends to connected finite simplicial complexes, see \cite{Babenko(2002)}, section 2. Hence, one may define stable systolic constants for simplicial complexes, which for triangulated manifolds coincide with the ones defined above. 

\begin{defn}
A simplicial map $f:X\to Y$ between simplicial complexes of dimension $n$ will be called \emph{$(n,d)$-monotone} if the preimage of every open $n$-simplex of $Y$ consists of at most $d$ open $n$-simplices in $X$.
\end{defn}

Consider the following two axioms for real-valued invariants $I$ of connected finite simplicial complexes.

\begin{comp_ax*}
Let $X$ and $Y$ be two connected finite simplicial complexes of dimension $n$. If there exists an $(n,d)$-monotone map $f:X\to Y$ such that the induced homomorphism $f_*:H_k(X;\mathbb{Z})_\mathbb{R}\hookrightarrow H_k(Y;\mathbb{Z})_\mathbb{R}$ is injective and the image $f_*(H_k(X;\mathbb{Z})_\mathbb{R})$ is contained in $r\cdot H_k(Y;\mathbb{Z})_\mathbb{R}$ for a positive integer $r$, then
\[ I(X) \leq d/r^{n/k} \cdot I(Y). \]
\end{comp_ax*}

\begin{ext_ax*}
Let $X$ be a connected finite $n$-dimensional simplicial complex, and let $X'$ be an \emph{extension} of $X$, i.\,e.\ $X'$ is obtained from $X$ by attachment of finitely many cells of dimension $1\leq \ell \leq n-1$ such that the inclusion $X\hookrightarrow X'$ induces the composition of a split monomorphism $H_k(X;\mathbb{Z})_\mathbb{R}\hookrightarrow H_k(X';\mathbb{Z})_\mathbb{R}$ with multiplication by some positive integer $r$. Then
\[ I(X') = r^{n/k}\cdot I(X). \]
\end{ext_ax*}

We will prove that both axioms are satisfied for the stable $k$-systolic constant. Analogous statements were shown in \cite{Babenko(2006)} for the $1$-systolic constant, in \cite{Sabourau(2006)} for the minimal entropy, and in \cite{Brunnbauer(2007b)} for the spherical volume. Similar ideas may be found in various papers on systolic invariants under the keywords `meromorphic map' or `$(n,k)$-morphism'. (See for example \cite{BK(1998)}, \cite{BKS(1998)}, \cite{KS(1999)}, and \cite{KS(2001)}.)

\begin{lem}
The comparison axiom holds for $I=\sigma^{st}_k$.
\end{lem}

See e.\,g.\ \cite{Babenko(2002)}, Proposition 2.2.7 for a similar argument. There, a kind of systolic freedom (i.\,e.\ the vanishing of a suitably defined systolic constant) is pulled back. Here, the stable systolic constant of $Y$ may also be nonzero. The used pullback technique goes back to \cite{Babenko(1992)}, Proposition 2.2.

\begin{proof}
Choose continuous piecewise smooth Riemannian metrics $g_1$ and $g_2$ on $X$ and $Y$ respectively. Then
\[ g_1^t := f^*g_2 + t^2 g_1 \]
with $t>0$ is again a Riemannian metric on $X$. Choosing $t>0$ small enough it can be arranged that 
\[ \Vol(X,g_1^t) \leq d\cdot \Vol(Y,g_2)+\varepsilon \]
for any given $\varepsilon>0$. Moreover,
\[ f:(X,g_1^t) \to (Y,g_2) \]
is $1$-Lipschitz and thus decreases the stable norm. Since $f_*:H_k(X;\mathbb{Z})_\mathbb{R}\hookrightarrow H_k(Y;\mathbb{Z})_\mathbb{R}$ is injective and since $f_*(H_k(X;\mathbb{Z})_\mathbb{R})\subset r\cdot H_k(Y;\mathbb{Z})_\mathbb{R}$, it follows that
\[ \stabsys_k(X,g_1^t) \geq r\cdot \stabsys_k(Y,g_2) \]
by the fact that the stable norm is a norm. Therefore, $\sigma_k^{st}(X) \leq d/r^{n/k}\cdot \sigma_k^{st}(Y)$.
\end{proof}

\begin{lem}\label{ext_ax}
The stable $k$-systolic constant satisfies the extension axiom.
\end{lem}

In \cite{BKSW(2006)}, Proposition 7.3, this is proved for a special case. Note also section 10 of the cited paper. With some adjustments the proof carries over to the general case. A similar argument was first used in \cite{BK(1998)}, Lemma 6.1. 

\begin{proof}
The inclusion $i:X\hookrightarrow X'$ is $(n,1)$-monotone, induces a monomorphism on the integral lattices of $k$-dimensional homology, and $i_*(H_k(X;\mathbb{Z})_\mathbb{R})\subset r\cdot H_k(X';\mathbb{Z})_\mathbb{R}$, hence $\sigma_k^{st}(X) \leq 1/r^{n/k}\cdot \sigma_k^{st}(X')$ by the comparison axiom.

To prove the converse inequality, we want to use induction over the number of attached cells. But it may happen that the attachment of some $k$-cells increases the $k$-th Betti number and that later on the attachment of some $(k+1)$-cells decreases it again. Then the induced homomorphism on real $k$-dimensional homology may be injective on the whole but it is not injective at every induction step. Therefore, it is useful to introduce the following `relative' version of the stable $k$-systolic constant.

\begin{defn}
Let $b$ be a positive integer, and let $\phi:H_k(X;\mathbb{Z})_\mathbb{R}\to \mathbb{Z}^b$ be a homomorphism. The induced homomorphism $H_k(X;\mathbb{R})\to\mathbb{R}^b$ will also be denoted by $\phi$. For a metric $g$ on $X$ the \emph{stable $(\phi,k)$-systole} $\stabsys_{\phi,k}(X,g)$ is defined as the minimum of the quotient norm of the stable norm on the nonzero elements of the lattice $\mathbb{Z}^b$ in $\mathbb{R}^b$. The \emph{stable $(\phi,k)$-systolic constant} is given by
\[ \sigma_{\phi,k}^{st}(X) := \inf_g \frac{\Vol(X,g)}{\stabsys_{\phi,k}(X,g)^{n/k}}. \]
\end{defn}

Note that for $\phi$ a split monomorphism, this definition coincides with the original `absolute' one.

Now, an \emph{extension} $(X',\phi')$ of $(X,\phi)$ consists of a simplicial complex $X'$ that is obtained from $X$ by attaching finitely many cells of dimension $1\leq\ell\leq n-1$ and of a homomorphism $\phi':H_k(X';\mathbb{Z})_\mathbb{R}\to \mathbb{Z}^b$ such that $\phi=\phi'\circ i_*$ with $i:X\hookrightarrow X'$ the inclusion. We will prove the following relative version of the extension axiom. 

\begin{claim*}
If $(X',\phi')$ is an extension of $(X,\phi)$, then $\sigma_{\phi',k}^{st}(X') \leq \sigma_{\phi,k}^{st}(X)$. 
\end{claim*}

In fact equality holds because a relative version of the comparison axiom is also fulfilled. Moreover, this claim implies that $\sigma_k^{st}$ satisfies the original extension axiom: taking $\phi'$ as an isomorphism and $\phi$ as the composition $\phi'\circ i_*$ one gets
\[ \sigma_k^{st}(X') = \sigma_{\phi',k}^{st}(X') \leq \sigma_{\phi,k}^{st}(X). \]
Furthermore, $\stabsys_k(X,g) = r\cdot\stabsys_{\phi,k}(X,g)$ since the stable norm is a norm. Therefore, $\sigma_{\phi,k}^{st}(X) = r^{n/k}\cdot\sigma_k^{st}(X)$ and the extension axiom follows.

With the relative version of extension it is possible to proceed by induction over the number of attached cells. To prove the claim it suffices therefore to consider the case where a single $\ell$-cell is attached to $X$. 

Note that the volume of $X'$ equals the volume of $X$ since the attached cell is of lower dimension, hence is a set of measure zero with respect to any $n$-dimensional volume.

Let $g$ be a Riemannian metric on $X$, and let $h:S^{\ell-1}\to X$ be the simplicial attaching map. Choose $R>0$ such that $h:(S^{\ell-1},g_R)\to (X,g)$ is nonexpanding where $g_R$ denotes the round metric of radius $R$. Define a Riemannian metric on $X'=X\cup_h D^\ell$ in the following way: think of $X'$ as divided into four pieces
\[ X \;\;\; \cup_h \;\;\; (S^{\ell-1}\times [-1,0]) \;\;\;\cup\;\;\; (S^{\ell-1}\times [0,L]) \;\;\; \cup \;\;\; S^\ell_+ \]
and take the following Riemannian metrics on the respective pieces
\[ g, \;\;\; ((1+t)g_R - t h^*g)\oplus dt^2, \;\;\; g_R\oplus dt^2, \;\;\; g_R, \]
where $(S^\ell_+, g_R)$ is an $\ell$-dimensional round hemisphere of radius $R$ and $L>0$ is some (large) number. This gives a Riemannian metric $g_L$ on $X'$.

If $\stabsys_{\phi',k}(X',g_L)\geq \stabsys_{\phi,k}(X,g)$ for some $L>0$, we are done. Hence, we may assume that $\stabsys_{\phi',k}(X',g_L) < \stabsys_{\phi,k}(X,g)$ for every $L>0$.

Let $\alpha'\in H_k(X';\mathbb{R})$ represent via $\phi$ a nonzero class in $\mathbb{Z}^b$ such that its stable norm $\|\alpha'\|=\stabsys_{\phi',k}(X',g_L)$, and let $c\in C_k(X';\mathbb{R})$ be a real cycle representing $\alpha'$ such that $\Vol_k(c)\leq \|\alpha'\|+\varepsilon$. 

Next, we apply the coarea formula to the projection $p$ of $S^{\ell-1}\times [0,L]$ to the second factor. Denote $c_t:= c\cap p^{-1}(t)$. Then
\[ \int_0^L \Vol_{k-1}(c_t) dt \leq \Vol_k(c), \]
and therefore there is a $t_0$ such that 
\[ \Vol_{k-1}(c_{t_0}) \leq \Vol_k(c)/L \leq (\|\alpha'\|+\varepsilon)/L. \]

Since the right hand side is bounded by $(\stabsys_{\phi,k}(X,g)+\varepsilon)/L$, we can force the volume of $c_{t_0}$ to be arbitrarily small by choosing $L$ very large. By the isomperimetric inequality for small cycles (see \cite{Gromov(1983)}, Sublemma 3.4.B') applied to $S^{\ell-1} \times t_0$ there is a filling $d$ of $c_{t_0}$ of volume
\[ \Vol_k(d) \leq C_{R,\ell} \cdot \Vol_{k-1}(c_{t_0})^{k/(k-1)}, \]
with a constant $C_{R,\ell}>0$ depending only on the radius $R$ and the dimension $\ell$. Assuming $\Vol_{k-1}(c_{t_0})\leq 1$, we get a `linear isoperimetric inequality':
\[ \Vol_k(d) \leq C_{R,\ell} \cdot \Vol_{k-1}(c_{t_0}). \]

The cycle $c$ decomposes into two pieces along $c_{t_0}$, that is to say $c=c_+\cup_{c_{t_0}} c_-$. Define another cycle
\[ c' := c_+ \cup_{c_{t_0}} d = c - (c_-\cup_{c_{t_0}} (-d)). \]
Since the cycle $c_-\cup_{c_{t_0}} (-d)$ is contained in the attached $\ell$-cell, it is null-homologous. Thus, $c'$ also represents $\alpha'$. Moreover,
\begin{align*}
\Vol_k(c') & \leq \|\alpha'\|+\varepsilon + C_{R,\ell}(\|\alpha'\|+ \varepsilon)/L \\
& = (\|\alpha'\|+\varepsilon) (1+C_{R,\ell}/L).
\end{align*}

The map that contracts the cylinder $S^{\ell-1}\times [-1,L]$ to $X$ is nonexpanding. Hence, the image $c''$ of $c'$ under this retraction satisfies the same volume bound and still represents $\alpha'$. 

The homology class $\alpha\in H_k(X;\mathbb{R})$ represented by $c''$ in $X$ is a preimage of $\alpha'$. Therefore, it represents a nonzero element of the lattice $\mathbb{Z}^b\subset \mathbb{R}^b$.  Moreover,
\begin{align*}
\|\alpha\| &\leq \Vol_k(c'') \\
&\leq (\|\alpha'\|+\varepsilon)(1+C_{R,\ell}/L) \\
&= (\stabsys_{\phi',k}(X',g_L)+\varepsilon)(1+C_{R,\ell}/L).
\end{align*}

Since $\varepsilon>0$ was chosen arbitrarily, we see that
\[ \stabsys_{\phi,k}(X,g) \leq \stabsys_{\phi',k}(X',g_L)(1+C_{R,\ell}/L). \]
For $L$ tending to infinity, this implies $\sigma_{\phi,k}^{st}(X) \geq \sigma_{\phi',k}^{st}(X')$. Thus the claim is proved, and the extension axiom is valid for $I=\sigma_k^{st}$.
\end{proof}

\section{Two topological lemmata}\label{top_lemmata}

In the proofs of the three theorems stated in section \ref{intro} we will frequently use the following two lemmata. Both are direct consequences of \cite{Brunnbauer(2007b)}, Lemmata 2.2 and 2.3.

\begin{lem}\label{n-1-skeleton}
Let $M$ be a connected closed manifold of dimension $n\geq 3$, and let $f:M\to X$ be a map to a CW complex that is surjective on fundamental groups. Then $f$ is homotopic to a map from $M$ to the $(n-1)$-skeleton of $X$ if and only if one of the following statements holds:
\begin{enumerate}
\item $M$ is orientable, and $f_*[M]_\mathbb{Z}=0$ in $H_n(X;\mathbb{Z})$.
\item $M$ is nonorientable, $f$ maps every orientation reversing loop to a noncontractible one, and $f_*[M]_\mathcal{O}=0$ in $H_n(X;\mathcal{O})$.
\item $M$ is nonorientable, $f$ maps some orientation reversing loop to a contractible one, and $f_*[M]_{\mathbb{Z}_2}=0$ in $H_n(X;\mathbb{Z}_2)$.
\end{enumerate}
\end{lem}

The local coefficient systems $\mathcal{O}$ in part (ii) are the orientation bundle on $M$ and the induced integer bundle on $X$. Recall that there is a bijection between the nontrivial integer bundles over a CW complex and the subgroups of its fundamental group of index two. For example, the orientation bundle $\mathcal{O}$ of $M$ corresponds to the subgroup of orientation preserving loops in $M$. If the kernel of the epimorphism $f_*:\pi_1(M)\twoheadrightarrow \pi_1(X)$ consists only of orientation preserving loops, then the image of the subgroup of orientation preserving loops has again index two and determines a nontrivial integer bundle $\mathcal{O}$ over $X$. (See \cite{Brunnbauer(2007b)}, section 2.2.)

\begin{lem}\label{n,1-monotone}
Let $M$ and $N$ be two connected closed orientable manifolds of dimension $n\geq 3$. Let $i:M\hookrightarrow X$ be an embedding into a CW complex, and let $f:N\to X$ be a map such that the induced homomorphism on fundamental groups is surjective and such that $f_*[N]_\mathbb{Z} = i_*[M]_\mathbb{Z}$ in $H_n(X;\mathbb{Z})$. Identify $M$ and its image in $X$. Then $f$ is homotopic to an $(n,1)$-monotone map $N\to M\cup X^{(n-1)}$.
\end{lem}

\section{Spheres and their loop spaces}\label{spheres_loop_sp}

In this section, we will briefly recall some topological properties of spheres and their loop spaces that will be used in the proofs of Theorem \ref{stabsys_thm} and Theorem \ref{freedom}.

If $k$ is odd, let $L_k$ denote the $k$-dimensional sphere $S^k$. The cohomology ring is the exterior algebra
\[ H^*(L_k;\mathbb{Z}) \cong \Lambda_\mathbb{Z}[\alpha] \]
with $\alpha\in H^k(L_k;\mathbb{Z})$ a generator. Therefore, by the K\"unneth formula
\[ H^*(L_k^b;\mathbb{Z}) \cong \Lambda_\mathbb{Z} [\alpha_1,\ldots,\alpha_b] \]
with $\alpha_i$ of degree $k$. Furthermore, we will always use the CW decomposition of $S^k$ consisting of one $0$-cell and one $k$-cell.

It is known by work of Serre (\cite{Serre(1951)}) that the homotopy groups $\pi_m(S^k)$ for $m>k$ are finite (recall that $k$ is odd). Sullivan showed in \cite{Sullivan(1974)} that the selfmap $S^k\to S^k$ of degree $d$ induces a homomorphism $\pi_m(S^k)\to \pi_m(S^k)$, which is nilpotent on $d$-torsion. Therefore, there exists a nonzero degree selfmap of $S^k$ that induces the trivial homomorphism on $\pi_m(S^k)$.

Finally note that the map $S^k\to S^k$ of degree two multiplies each homology class in $L_k^b$ of dimension $\ell>0$ by some power of two. In particular, the homomorphism on homology with $\mathbb{Z}_2$ coefficients is trivial.

If $k$ is even, let $L_k$ be a CW complex that is homotopy equivalent to the based loop space $\Omega S^{k+1}$ of the $(k+1)$-dimensional sphere. More precisely, let $L_k$ be the James reduced product $J(S^k)$. (See for example \cite{Hatcher(2002)}, pages 224--225 and section 4.J.) The CW structure on $L_k$ consists of one cell in each dimension divisible by $k$, and the cohomology ring is the divided polynomial algebra
\[ H^*(L_k;\mathbb{Z}) \cong \Gamma_\mathbb{Z}[\alpha] \]
with $\alpha\in H^k(L_k;\mathbb{Z})$. (Recall that this is almost a polynomial algebra. In fact, the generator in degree $kp$ equals $\alpha^p/p!$. Using real coefficients, the cohomology ring is a polynomial algebra.) By the K\"unneth formula
\[ H^*(L_k^b;\mathbb{Z}) \cong \Gamma_\mathbb{Z}[\alpha_1,\ldots,\alpha_b], \]
where the $\alpha_i$ are of degree $k$.

Since $\pi_m(\Omega S^{k+1})=\pi_{m+1}(S^{k+1})$, the homotopy groups $\pi_m(L_k)$ are finite for $m>k$ (recall that $k$ is even now). The group structure in $\pi_m(L_k)$ coincides with the one coming from loop multiplication, hence one can easily construct selfmaps $L_k\to L_k$ that induce multiplication by some positive integer on $H^k(L_k;\mathbb{Z})$ and trivial homomorphisms on $\pi_m(L_k)$.

The map $\Omega S^{k+1}\to \Omega S^{k+1}$ that assigns to each loop its double induces multiplication by two in $\pi_k(L_k^b)$ and therefore also in $H_k(L_k^b;\mathbb{Z})$ and in $H^k(L_k^b;\mathbb{Z})$. Hence it induces multiplication by some power of two in every $H^\ell(L_k^b;\mathbb{Z})$ with $\ell>0$ and thus in every $H_\ell(L_k^b;\mathbb{Z})$, as well. Consequently, the induced homomorphism on homology with $\mathbb{Z}_2$ coefficients is trivial. 

\begin{summ*} 
The CW complexes $L_k^b$ have cells only in dimensions divisible by $k$, their real cohomology rings are generated by elements of degree $k$, and there exist selfmaps $h_m:L_k^b\to L_k^b$ for all $m>k$ that induce multiplication by some positive integer on $H^k(L_k^b;\mathbb{Z})$ (and therefore also on $H_k(L_k^b;\mathbb{Z})$) and vanish on $\pi_m(L_k^b)$. Moreover, there exists a selfmap $L_k^b\to L_k^b$ that induces the zero homomorphism in homology of nonzero dimension with coefficients in $\mathbb{Z}_2$ but that is bijective on $H_k(L_k^b;\mathbb{R})$.
\end{summ*}

\section{Existence of stable systolic inequalities}\label{existence}

Using the properties of $L_k$ listed in the preceding section, we are able to prove:

\begin{lem}\label{map_loopspace_sphere}
Let $X$ be a connected CW complex of dimension $n$. Let $1\leq k\leq n-1$, and assume that the $k$-th Betti number $b:=b_k(X)$ is finite. Then there is a map $f:X\to L_k^b$ that induces an isomorphism 
\[ f_*: H_k(X;\mathbb{R}) \xrightarrow{\cong} H_k(L_k^b;\mathbb{R}). \]
In fact, $f$ can be chosen such that the induced map on the integral lattices corresponds to multiplication by some positive integer.
\end{lem}

This lemma stems from \cite{KS(1999)}, where the map $f$ is constructed in sections 4 and 10.

\begin{proof}
The case $k=1$ is easy because $L_1^b$ is just the $b$-dimensional torus $T^b$, which is $K(\mathbb{Z}^b,1)$. The canonical epimorphism
\[ \pi_1(X) \twoheadrightarrow H_1(X;\mathbb{Z}) \twoheadrightarrow H_1(X;\mathbb{Z})_\mathbb{R} \]
is induced by the Jacobi map $f:X\to T^b$. The induced homomorphism $f_*:H_1(X;\mathbb{R})\to H_1(T^b;\mathbb{R})$ is consequently an isomorphism, which is moreover an isomorphism of the integral lattices.

Now, let $2\leq k \leq n-1$. Choose a CW decomposition of $K(\mathbb{Z}^b,k)$ such that the $(k+1)$-skeleton is $\bigvee^b S^k$, the wedge of $b$ spheres of dimension $k$. As in the introduction there is a map $X\to K(\mathbb{Z}^b,k)$ that induces an isomorphism on the integral lattices of homology in dimension $k$. By cellular approximation this gives a map $X^{(k+1)}\to \bigvee^b S^k$. Note that the $(k+1)$-skeleton of $L_k^b$ is also $\bigvee^b S^k$. Thus, we have a map 
\[ f^{(k+1)}:X^{(k+1)}\to L_k^b \] 
that induces an isomorphism on the integral lattices of homology in degree $k$.

Let $h_{k+1}:L_k^b\to L_k^b$ be as in the summary above, i.\,e.\ it induces the trivial homomorphism on the $(k+1)$-dimensional homotopy groups and multiplication by some positive integer on real homology of degree $k$. Then the  composition $h_{k+1}\circ f^{(k+1)}:X^{(k+1)}\to L_k^b$ extends over $X^{(k+2)}$ since it is zero on the $(k+1)$-dimensional homotopy groups. Call this extension $f^{(k+2)}:X^{(k+2)}\to L_k^b$. Repeating this process gives finally a map
\[ f:X\to L_k^b \]
for which the induced monomorphism $H_k(X;\mathbb{Z})_\mathbb{R}\hookrightarrow H_k(L_k^b;\mathbb{Z})_\mathbb{R}$ corresponds to multiplication by some positive integer.
\end{proof}

\begin{proof}[Proof of Theorem \ref{stabsys_thm}]
Denote $b:=b_k(M)$. If $k$ does not divide $n$, then the $n$-skeleton and the $(n-1)$-skeleton of $L_k^b$ coincide, and $\sigma_k^{st}(M)=0$ by the comparison axiom applied to the map $f:M\to (L_k^b)^{(n-1)}$ of Lemma \ref{map_loopspace_sphere} (which is $(n,0)$-monotone due to the dimension of the range). Assume now that $n=kp$.

If $n=2$, then $k=1$ and it is known that, apart from the sphere, all closed orientable surfaces satisfy a stable systolic inequality \eqref{stabsys_ineq} and that their cohomology rings are generated in degree one. Therefore, Theorem \ref{stabsys_thm} is true in this case and we may restrict our attention to $n\geq 3$.

Let $f:M\to L_k^b$ induce an isomorphism on real homology of degree $k$. If $f_*[M]_\mathbb{Z}=0$ in $H_n(L_k^b;\mathbb{R})$, then $f_*[M]_\mathbb{Z}=0$ also in $H_n(L_k^b;\mathbb{Z})$ since the homology of $L_k^b$ is torsion-free. By Lemma \ref{n-1-skeleton} one can homotope $f$ so that its image lies in the $(n-1)$-skeleton of $L_k^b$ (i.\,e.\ $f$ is $(n,0)$-monotone), and by the comparison axiom $\sigma_k^{st}(M)$ vanishes. The theorem of Gromov stated in the introduction (Theorem \ref{thm_gromov}) shows that there cannot be cohomology classes $\beta_1,\ldots,\beta_p\in H^k(M;\mathbb{R})$ having nonvanishing product.

On the other hand, if $f_*[M]_\mathbb{Z}\neq 0$, then there are cohomology classes $\beta'_1,\ldots,\beta'_p\in H^k(L_k^b;\mathbb{R})$ such that 
\[ \langle \beta'_1 \smile\ldots\smile \beta'_p, f_*[M]_\mathbb{Z} \rangle\neq 0 \]
since the cohomology of $L_k^b$ is generated by classes of degree $k$. Therefore, the cohomology classes $\beta_i := f^*\beta'_i \in H^k(M;\mathbb{R})$ have nonvanishing product and the stable $k$-systolic constant of $M$ is nonzero by Theorem \ref{thm_gromov}. This finishes the proof. 
\end{proof}

\begin{proof}[Proof of Theorem \ref{freedom}]
Let $f_k: M\to L_k^{b_k(M)}$ be as in Lemma \ref{map_loopspace_sphere}. Define
\[ F:=(f_2,\ldots,f_{n-1}) : M\longrightarrow L:= \prod_{k=2}^{n-1} L_k^{b_k(M)} .\]
Then by the K\"unneth formula, $F$ induces monomorphisms on $k$-dimensional real homology for all $k=2,\ldots,n-1$. 

There is a selfmap $L\to L$ that maps all homology classes (of nonzero dimension) to an even multiple. Composing $F$ with this map, one gets a map (still called $F$) that is injective on real homology of dimensions $k=2,\ldots,n-1$ and that is zero on $H_n(M;\mathbb{Z}_2)$, i.\,e.\ $F_*[M]_{\mathbb{Z}_2}=0$ in $H_n(L;\mathbb{Z}_2)$. By Lemma \ref{n-1-skeleton} it is possible to deform $F$ so that its range lies in the $(n-1)$-skeleton of $L$. (Note that $L$ is simply-connected and that we may assume $n\geq 3$ since the theorem is empty for $n=2$.)

Choose (continuous piecewise smooth) Riemannian metrics $g_1$ and $g_2$ on $M$ and the $(n-1)$-skeleton $L^{(n-1)}$ of $L$, respectively. Then
\[ g_1^t := F^*g_2 + t^2 g_1 \]
with $t>0$ is again a Riemannian metric on $M$. Choosing $t>0$ small enough, it can be arranged that 
\[ \Vol(M,g_1^t) \leq \varepsilon \]
for any given $\varepsilon>0$. Moreover,
\[ F:(M,g_1^t) \to (L^{(n-1)},g_2) \]
is $1$-Lipschitz. Since $F_*:H_k(M;\mathbb{Z})_\mathbb{R}\hookrightarrow H_k(L^{(n-1)};\mathbb{Z})_\mathbb{R}$ is injective for all $k=2,\ldots,n-1$, it follows that
\[ \stabsys_k(M,g_1^t) \geq \stabsys_k(L^{(n-1)},g_2). \]
Therefore, the left-hand side of
\[ \prod_{k=2}^{n-1} \stabsys_k(M,g_1^t)^{p_k} \leq C \cdot \Vol(M,g_1^t) \]
is bounded from below by the constant $\prod_k \stabsys_k(L^{(n-1)},g_2)^{p_k}$, whereas the right-hand side can be made arbitrarily small. Thus, there is no constant $C>0$ such that this inequality is satisfied for all metrics on $M$.
\end{proof}

The \emph{$k$-systole modulo torsion} $\sys_k^\infty(M,g)$ of a connected closed Riemannian manifold is defined as the infimum of the volume $\Vol_k$ on the nonzero classes in $H_k(M;\mathbb{Z})_\mathbb{R}$. The corresponding \emph{$k$-systolic constant modulo torsion} is denoted by $\sigma_k^\infty(M)$. 

By work of Katz and Suciu (\cite{KS(1999)} and \cite{KS(2001)}) it is known that $\sigma_k^\infty(M)$ vanishes for all $2\leq k\leq n-1$. (This phenomenon is called \emph{systolic freedom modulo torsion}.) For $k=1$, Gromov and Babenko showed (see \cite{Gromov(1983)}, Appendix 2, (B'$_1$) and \cite{Babenko(1992)}, Lemma 8.4) that $\sigma_1^\infty(M)$ is zero if and only if the Jacobi map $\Phi:M\to T^b$ is homotopic to a map with range in the $(n-1)$-skeleton of the torus.

Since $\sigma_1^{st}(M)\geq \sigma_1^\infty(M)$ by definition, and since Babenko's lemma also applies to the stable systolic constant, the same characterization is true for the vanishing of $\sigma_1^{st}(M)$. Using Lemma \ref{n-1-skeleton}, we get the following corollary.

\begin{cor}\label{stable-1-sys}
The stable $1$-systolic constant $\sigma_1^{st}(M)$ and the $1$-systolic constant modulo torsion $\sigma_1^\infty(M)$ vanish if and only if the fundamental class of $M$ with coefficients in $\mathbb{Z}$, $\mathcal{O}$, or $\mathbb{Z}_2$ (according to the orientation behaviour of the Jacobi map) is mapped to zero by the Jacobi map.
\end{cor}

For surfaces this can easily be seen directly. Note in particular that the only surfaces with zero stable $1$-systolic constant are the sphere, the real projective plane, and the Klein bottle.

\section{Homological invariance of stable systolic constants}\label{homol_invar}

The first aim of this section is to prove Theorem \ref{homol_invar_thm}. Then, we will apply the theorem to the case of projective spaces over division algebras. Afterwards an analogous theorem for the nonorientable case will be stated, where $k=1$ may be assumed according to Theorem \ref{freedom}. 

\begin{proof}[Proof of Theorem \ref{homol_invar_thm}]
The images of the fundamental classes of $M$ and $N$ in the integral homology group $H_n(K(\mathbb{Z}^b,k);\mathbb{Z})$ may differ by a torsion class. Thus, Lemma \ref{n,1-monotone} cannot be applied directly. But we can bypass this problem in the following way.

Let $f:K(\mathbb{Z}^b,k)^{(n+1)} \to L_k^b$ be a map as in Lemma \ref{map_loopspace_sphere}. (Since $K(\mathbb{Z}^b,k)$ is not finite-dimensional in general, we have to consider some skeleton to be able to apply Lemma \ref{map_loopspace_sphere}.) In particular, the induced map on the integral lattices of $k$-dimensional homology corresponds to multiplication by some positive integer $r$. Then $f_*\Phi_*[M]_\mathbb{Z} = f_*\Psi_*[N]_\mathbb{Z}$ in $H_n(L_k^b;\mathbb{Z})$ since the homology of $L_k^b$ is torsion-free.

Using the mapping cylinder of $f\circ\Phi$, we may assume without loss of generality that $f\circ\Phi$ is the inclusion of a subcomplex into $L_k^b$. If we start with a CW decomposition of $M$ with only one $0$-cell, then there is a $1$-cell in the mapping cylinder connecting this $0$-cell with the $0$-cell of $L_k^b$. Collapsing this $1$-cell, we may assume that $L_k^b$ is obtained from $M$ by attaching only cells of positive dimension.

By Lemma \ref{n,1-monotone}, $f\circ \Psi$ is homotopic to an $(n,1)$-monotone map
\[ \Psi':N\longrightarrow X:= M\cup (L_k^b)^{(n-1)}. \] 
Note that $X$ is a finite complex. In particular, it is an extension of $M$. By the comparison and extension axiom, we get
\[ \sigma_k^{st}(N) \leq 1/r^{n/k}\cdot\sigma_k^{st}(X) = \sigma_k^{st}(M). \]
Changing the roles of $M$ and $N$ gives equality.
\end{proof}

\begin{rem*}
In the cases $k=1$ and $k=2$, it is not necessary to use a map $f:K(\mathbb{Z}^b,k)\to L_k^b$ because the homology of $K(\mathbb{Z},1)=S^1$ and $K(\mathbb{Z},2)=\mathbb{C}\mathrm{P}^\infty$ is torsion-free.
\end{rem*}

A direct consequence of Theorem \ref{homol_invar_thm} is the following corollary.

\begin{cor}
Let $f:M\to N$ be a degree one map between connected closed orientable manifolds such that the induced homomorphism 
\[ f_* :H_k(M;\mathbb{Z})_\mathbb{R} \xrightarrow{\cong} H_k(N;\mathbb{Z})_\mathbb{R} \]
is bijective. Then $\sigma_k^{st}(M) = \sigma_k^{st}(N)$.
\end{cor}

For example, the inclusion $\mathbb{C}^{2n+1}\times 0 \hookrightarrow \mathbb{C}^{2n+2}=\mathbb{H}^{n+1}$ gives a degree one map $\mathbb{C}\mathrm{P}^{2n}\to \mathbb{H}\mathrm{P}^n$ from complex to quaternionic projective space that induces isomorphisms of homology in all dimensions divisible by $4$. Therefore, the corollary implies
\[ \sigma_{4k}^{st}(\mathbb{H}\mathrm{P}^n) = \sigma_{4k}^{st}(\mathbb{C}\mathrm{P}^{2n}), \]
which is nonzero if and only if $k$ divides $n$ by Theorem \ref{stabsys_thm}. In particular, 
\[ \sigma_{4n}^{st}(\mathbb{H}\mathrm{P}^{2n}) = \sigma_{4n}^{st}(\mathbb{C}\mathrm{P}^{4n}). \]
This last equation is Theorem 1.2 from \cite{BKSW(2006)}. However, note that our proof is not essentially different from the proof of Bangert, Katz, Shnider, and Weinberger. In fact, their reasoning is along the same lines but only for the special case of complex and quaternionic projective spaces.

Note also that in dimension eight, they show that
\[ \sigma_4^{st}(\mathbb{H}\mathrm{P}^2) = \sigma_4^{st}(\mathbb{C}\mathrm{P}^4) \in [{\textstyle\frac{1}{14},\frac{1}{6}}]. \]
In fact, the calculation of this stable systolic constant is possibly within reach, see \cite{BKSW(2006)}.

For the rest of this section let $k=1$. For orientable manifolds (\emph{first type}) integer coefficients will be used. If $M$ is nonorientable and the subgroup of orientation preserving loops contains the kernel of the canonical epimorphism $\phi:\pi_1(M)\twoheadrightarrow H_1(M;\mathbb{Z})_\mathbb{R}$ (\emph{second type}), then we will use the orientation bundle $\mathcal{O}$ as (local) coefficients for homology. Moreover, we will consider only maps $\Phi:M\to T^b$ such that only the last circle of $T^b=S^1\times\cdots\times S^1$ is represented by an orientation reversing loop in $M$ and all other factors by orientation preserving loops. (Then the induced coefficient systems $\mathcal{O}$ on $T^b$ coincide for all such maps $\Phi$ from different manifolds of second type.) If $M$ is nonorientable and there are orientation reversing loops in the kernel of $\phi$ (\emph{third type}), then we will use coefficients in $\mathbb{Z}_2$.

\begin{thm}
Let $M$ and $N$ be two connected closed manifolds with $b_1(M)=b_1(N)=:b$, and let $\Phi:M\to T^b$ and $\Psi:N\to T^b$ be maps that induce isomorphisms on first cohomology with integer coefficients (with the mentioned restriction for second type manifolds).
\begin{enumerate}
\item If both manifolds are of the same type and $\Phi_*[M]_\mathbb{K}=\Psi_*[N]_\mathbb{K}$ (where $\mathbb{K}$ stands for $\mathbb{Z}$, $\mathcal{O}$, or $\mathbb{Z}_2$ according to the type), then
\[ \sigma_1^{st}(M) = \sigma_1^{st}(N). \]
\item If $N$ is of the third type and $\Phi_*[M]_{\mathbb{Z}_2}=\Psi_*[N]_{\mathbb{Z}_2}$, then
\[ \sigma_1^{st}(M) \geq \sigma_1^{st}(N). \]
\end{enumerate}
\end{thm}

This is a special case of \cite{Brunnbauer(2007b)}, Theorem 7.1 for $I=\sigma_1^{st}$. (This theorem is valid only for manifolds of dimension $n\geq 3$. But the case of surfaces can easily be verfied directly.)

\section{The multilinear intersection form}\label{cohomol_ring}

Consider two simply-connected closed four-manifolds $M$ and $N$. If their intersection forms are equivalent, then the manifolds are homotopy equivalent by the theorem of Milnor and Whitehead. Therefore, their stable $2$-systolic constants coincide. (For this only the comparison axiom is needed since every homotopy equivalence of manifolds is homotopic to an $(n,1)$-monotone map by Lemma \ref{n,1-monotone}.) We will see that one can drop the assumption on the fundamental group using Theorem \ref{homol_invar_thm}.

More generally, let $M$ be a connected closed manifold of dimension $kp$. Consider the \emph{multilinear intersection form}
\begin{align*} 
Q^k_M:(H^k(M;\mathbb{Z})_\mathbb{R})^p &\to \mathbb{Z},\\ 
(\beta_1,\ldots,\beta_p) &\mapsto \langle \beta_1\smile\cdots\smile\beta_p,[M]_\mathbb{Z} \rangle.
\end{align*}
In the case $p=2$, this is the usual intersection form. 

Before we start with the proof of Corollary \ref{intersection_form}, note that a similar result was derived by Hamilton (see \cite{Hamilton(2006)}, Theorem 1.2): if two closed orientable four-manifolds with $b_2^+=1$ have equivalent intersection forms, then their conformal systolic constants agree. Here, the \emph{conformal systolic constant} $\mathrm{CS}(M)$ is the supremum of the \emph{conformal systole} over all Riemannian metrics on $M$. (See \cite{Katz(2007)} for more details on conformal systoles.)

\begin{proof}[Proof of Corollary \ref{intersection_form}]
Write  $b:= b_k(M)= b_k(N)$. Choose maps $\Phi:M\to K(\mathbb{Z}^b,k)$ and $\Psi:N\to K(\mathbb{Z}^b,k)$ inducing isomorphisms on the integral lattices of $k$-dimensional cohomology such that the isomorphism
\[ \Psi^*\circ (\Phi^*)^{-1} : H^k(M;\mathbb{Z})_\mathbb{R} \xrightarrow{\cong} H^k(N;\mathbb{Z})_\mathbb{R} \]
is an equivalence of the multilinear forms $Q^k_M$ and $Q^k_N$. Then
\begin{align*}
\langle \beta_1\smile\cdots\smile\beta_p, \Phi_*[M]_\mathbb{Z} \rangle &= \langle \Phi^*\beta_1\smile\cdots\smile\Phi^*\beta_p, [M]_\mathbb{Z} \rangle \\
&= \langle \Psi^*\beta_1\smile\cdots\smile\Psi^*\beta_p, [N]_\mathbb{Z} \rangle \\
&= \langle \beta_1\smile\cdots\smile\beta_p, \Psi_*[N]_\mathbb{Z} \rangle
\end{align*}
for all cohomology classes $\beta_1,\ldots,\beta_p$ in $H^k(K(\mathbb{Z}^b,k);\mathbb{R})$. 

Since the cohomology ring of the Eilenberg-Mac\,Lane space $K(\mathbb{Z}^b,k)$ equals the exterior algebra $\Lambda_\mathbb{R}[\alpha_1,\ldots,\alpha_b]$ if $k$ is odd and the polynomial algebra $\mathbb{R}[\alpha_1,\ldots,\alpha_b]$ if $k$ is even by work of Cartan and Serre (see for example \cite{Whitehead(1978)}, page 670), it follows that the classes $\Phi_*[M]_\mathbb{Z}$ and $\Psi_*[N]_\mathbb{Z}$ coincide in $H_n(K(\mathbb{Z}^b,k);\mathbb{R})$. By Theorem \ref{homol_invar_thm}, the stable $k$-systolic constants of $M$ and $N$ are equal.
\end{proof}

Consider the octonionic projective plane $\mathbb{O}\mathrm{P}^2$. Its cohomology ring is isomorphic to $\mathbb{Z}[\alpha]/(\alpha^3)$ with $\alpha$ of degree eight. Thus, the intersection form $Q^8_{\mathbb{O}\mathrm{P}^2}$ on eight-dimensional cohomology is given by
\begin{align*}
Q^8_{\mathbb{O}\mathrm{P}^2} : H^8(\mathbb{O}\mathrm{P}^2;\mathbb{Z})\times &H^8(\mathbb{O}\mathrm{P}^2;\mathbb{Z}) \to \mathbb{Z}, \\
(\alpha &, \alpha) \longmapsto 1.
\end{align*}
The intersection forms of $\mathbb{C}\mathrm{P}^8$ and $\mathbb{H}\mathrm{P}^4$ on the respective eight-dimensional cohomology groups are obviously equivalent over $\mathbb{Z}$ to the intersection form of the octonionic projective plane. Therefore, 
\[ \sigma_8^{st}(\mathbb{O}\mathrm{P}^2) = \sigma_8^{st}(\mathbb{H}\mathrm{P}^4) = \sigma_8^{st}(\mathbb{C}\mathrm{P}^8). \]

Note that this shows that neither the canonical metric of the octonionic projective plane nor the symmetric metric of the quaternionic projective four-space are systolically optimal. In fact, Berger showed in \cite{Berger(1972b)} that
\begin{align*} 
\Vol(\mathbb{C}\mathrm{P}^8, g_0)/\stabsys_8(\mathbb{C}\mathrm{P}^8, g_0)^2 &= 1/70, \\
\Vol(\mathbb{H}\mathrm{P}^4, g_0)/\stabsys_8(\mathbb{H}\mathrm{P}^4, g_0)^2 &= 5/126, \\
\Vol(\mathbb{O}\mathrm{P}^2, g_0)/\stabsys_8(\mathbb{O}\mathrm{P}^2, g_0)^2 &= 7/66, 
\end{align*}
where $g_0$ denotes the respective canonical Riemannian metrics.


\bibliographystyle{amsalpha}
\bibliography{asymp_invar}  

\end{document}